\documentclass[amsfonts,12pt]{amsart}
\usepackage{epsfig}
\usepackage{amsmath,amssymb}
\newtheorem{thm}{Theorem}[section]

\newtheorem{lem}[thm]{Lemma}
\newtheorem{cor}[thm]{Corollary}
\newtheorem{rem}[thm]{Remark}
\newtheorem{prp}[thm]{Proposition}

\newcommand{\conv}{{\mathrm{conv}}\,}

\newcommand{\ee}{\varepsilon}

\newcommand{\R}{\mathbb{R}}

\newcommand{\N}{\mathbb{N}}

\newcommand{\cS}{\mathcal{S}}

\newcommand{\Spm}{S_{\varphi,\mu}}

\newcommand{\SpAm}{S_{\varphi,A\mu}}

\begin{document}
\hfill\date{}
\bigskip

\title[THE DUAL ORLICZ-BRUNN-MINKOWSKI THEORY]{THE DUAL ORLICZ-BRUNN-MINKOWSKI THEORY}

\author[Richard J. Gardner, Daniel Hug, Wolfgang Weil, and Deping Ye]
{Richard J. Gardner, Daniel Hug, Wolfgang Weil, and Deping Ye}
\address{Department of Mathematics, Western Washington University,
Bellingham, WA 98225-9063, USA} \email{richard.gardner@wwu.edu}
\address{Karlsruhe Institute of Technology, Department of Mathematics,
D-76128 Karlsruhe, Germany}
\email{daniel.hug@kit.edu}
\address{Karlsruhe Institute of Technology, Department of Mathematics,
D-76128 Karlsruhe, Germany} \email{wolfgang.weil@kit.edu}
\address{Department of Mathematics and Statistics, Memorial University of Newfoundland,\newline St.~John's, Newfoundland, Canada A1C 5S7} \email{deping.ye@mun.ca}
\thanks{First author supported in
part by U.S.~National Science Foundation Grant DMS-1402929.   Second and third authors supported in part by German Research Foundation (DFG) grants HU 1874/4-2 and WE 1613/2-2. Fourth author supported in part by an NSERC grant.}
\subjclass[2010]{Primary: 52A20, 52A30; secondary: 52A39, 52A40} \keywords{compact convex set, star set, star body, Brunn-Minkowski theory, Orlicz-Brunn-Minkowski theory, Minkowski addition, $L_p$ addition, $M$-addition, Orlicz addition, radial addition, Brunn-Minkowski inequality, Minkowski's first inequality}

\pagestyle{myheadings}
\markboth{RICHARD J. GARDNER, DANIEL HUG,
WOLFGANG WEIL, AND DEPING YE}{THE ORLICZ-BRUNN-MINKOWSKI THEORY}

\begin{abstract}
This paper introduces the dual Orlicz-Brunn-Minkowski theory for star sets. A radial Orlicz addition of two or more star sets is proposed and a corresponding dual Orlicz-Brunn-Minkowski inequality is established. Based on a radial Orlicz linear combination of two star sets, a formula for the dual Orlicz mixed volume is derived and a corresponding dual Orlicz-Minkowski inequality proved.  The inequalities proved yield as special cases the precise duals of the conjectured log-Brunn-Minkowski and log-Minkowski inequalities of B\"{o}r\"{o}czky, Lutwak, Yang, and Zhang.  A new addition of star sets called radial $M$-addition is also introduced and shown to relate to the radial Orlicz addition.
\end{abstract}

\maketitle

\section{Introduction}
The combination of Minkowski addition and volume leads to the rich and powerful classical Brunn-Minkowski theory for compact convex sets, which constitutes the core of modern convex geometry. Important results such as the Brunn-Minkowski inequality and Minkowski's first inequality play fundamental roles in attacking problems in analysis, geometry, quantum information theory, random matrices, and many other fields. Readers are referred to the excellent treatise by Schneider \cite{Sch93} for more information and references.

In the same spirit, the combination of radial addition and volume produces a corresponding theory for star sets called the dual Brunn-Minkowski theory. (See Section~\ref{prelim} for definitions.) This was initiated by Lutwak \cite{L1}, who also took a further major step in \cite{L5}, where several important concepts and fundamental results in the classical Brunn-Minkowski theory were provided with dual counterparts. For instance, the dual Minkowski inequality for the dual mixed volume is analogous to Minkowski's first inequality for the mixed volume, and plays a key role in the solution of the Busemann-Petty problem (see \cite{G3, GKS, Z1}), as do intersection bodies, the notion dual to projection bodies.  The dual Brunn-Minkowski theory has connections and applications to integral geometry, Minkowski geometry, the local theory of Banach spaces, geometric tomography, and stereology; see \cite{Gar06} and the references given there. The literature is large and continues to grow.  See, for example, \cite{Bernig2014, GardnerDP2014, FNRZ, Gardner2007, GHW, GZ, KPZ, Milman2006}.

One way to extend the classical Brunn-Minkowski theory and its dual is to replace the linear function $\varphi(t)=t$ (note that both Minkowski and radial addition are linear) by $\varphi(t)=t^p$.  When $p\ge 1$, Minkowski addition of convex bodies becomes $L_p$ addition, introduced by Firey \cite{Firey1961, Firey1962} and when $p\neq 0$, radial addition of star sets becomes $p$th radial addition. The combination of these additions with volume leads to the $L_p$-Brunn-Minkowski theory for convex bodies and its dual.  However, the $L_p$-Brunn-Minkowski theory only began in earnest with the ground-breaking paper of Lutwak \cite{L3}, after which it has had an enormous impact, providing stronger affine isoperimetric inequalities than the classical Brunn-Minkowski theory and strengthening links with information theory.  We refer the reader to the introductions in \cite{GHW, GHW2} and to \cite[Chapter~9]{Sch93} for more information and references.

The most recent extension of the classical Brunn-Minkowski theory is the new Orlicz-Brunn-Minkowski theory, with the homogeneous function $t^p$ replaced by a generally nonhomogeneous function $\varphi(t)$.  The Orlicz-Brunn-Minkowski theory for convex bodies was launched by Lutwak, Yang, and Zhang \cite{LYZ7, LYZ8} with affine isoperimetric inequalities for Orlicz centroid and projection bodies.  The lack of homogeneity of the function $\varphi(t)$ meant that the problem of defining a corresponding Orlicz addition of convex bodies remained, but this obstacle was overcome by Gardner, Hug, and Weil \cite{GHW2}.  (It turns out, mainly as a consequence of results obtained in \cite{GHW}, that Orlicz addition is not associative unless it is already $L_p$ addition.)  These authors also provide a general framework for the Orlicz-Brunn-Minkowski theory, derive formulas for the Orlicz mixed volume of two convex bodies, and prove Orlicz-Brunn-Minkowski and Orlicz-Minkowski inequalities, whose classical counterparts (the  Brunn-Minkowski inequality and Minkowski's first inequality) have numerous applications in many fields.  (Some of these results were obtained independently by Xi, Jin, and Leng \cite{XJL}.)  The new theory has already attracted considerable interest; see, for example, \cite{bor2013, BLYZ, bor2013-2, Chen2011, HLYZ, HabP, Li2011, Ludwig2009, Ye2012, Ye2013, Ye2014, Zhu2012}.
	
This paper aims to provide the basic setting for the dual Orlicz-Brunn-Minkowski theory for star sets.  In some respects, this is more delicate than the Orlicz-Brunn-Minkowski theory for convex bodies, partly due to the various flavors of star sets that have to be considered. In Section~\ref{Onorms}, radial Orlicz addition for two or more star sets is introduced and its basic properties are established.  Like Orlicz addition, the new radial Orlicz addition, defined by (\ref{Orldef}) below and denoted by $\widetilde{+}_{\varphi}$, enjoys several useful properties such as continuity and $GL(n)$ convariance, but it is associative only when it is $p$th radial addition; see Corollary~\ref{nonassoc}.  In Section~\ref{DBMI}, we prove a dual Orlicz-Brunn-Minkowski inequality, a special case of which is as follows.

{\em If $\varphi\in \Phi_2$ and ${\varphi_0}(x_1, x_2)=\varphi(x_1^{1/n},  x_2^{1/n})$ is concave, then for star sets $K$ and $L$ in $\R^n$ with $V_n(K)+V_n(L)>0$,
$$\varphi\left(\left(\frac{V_n(K)}{V_n(K\widetilde{+}_{\varphi}L)}\right)^{1/n},
\left(\frac{V_n(L)}{V_n(K\widetilde{+}_{\varphi}L)}\right)^{1/n}\right)\ge 1,$$
while if $\varphi_0$ is convex, the inequality is reversed.  If $\varphi_0$ is strictly concave (or convex, as appropriate) and $K$ and $L$ are star bodies with positive radial functions, then equality holds if and only if $K$ and $L$ are dilatates.}

Here $V_n$ denotes $n$-dimensional Hausdorff measure and $\Phi_m$, $m\in \N$, is the class of continuous functions $\varphi:[0,\infty)^m\to [0,\infty)$ that are strictly increasing in each component and such that $\varphi(o)=0$ and $\lim_{t\to \infty} \varphi(tx) =\infty$, for each $x\in [0,\infty)^m\setminus\{o\}$.  A similar result holds for $\varphi$ in a certain class $\Psi_2$ of functions that decrease in each component.

In Section~\ref{radOlc}, we develop a formula for the dual Orlicz mixed volume of star sets $K$ and $L$, denoted by $\widetilde{V}_\varphi (K,L)$, based on a definition of radial Orlicz linear combination.  This appears in the following dual Orlicz-Minkowski inequality proved in Theorem~\ref{maindom}.

{\em Let $\varphi:(0,\infty)\to \R$ be such that $\varphi_0(t)=\varphi(t^{1/n})$, $t>0$, is concave.  Suppose that $K$ and $L$ are star bodies in $\R^n$ with positive radial functions. Then
$$
\widetilde{V}_\varphi (K,L)\le V_n(K)\,\varphi\left(\left( \frac{V_n(L)}{V_n(K)} \right)^{1/n}\right),
$$
while if $\varphi_0$ is convex, the inequality is reversed.  If $\varphi_0$ is strictly concave (or convex, as appropriate), then equality holds if and only if $K$ and $L$ are dilatates.}

This dual Orlicz-Minkowski inequality is fundamental in establishing Orlicz affine isoperimetric inequalities for the dual Orlicz affine and geominimal surface areas; see \cite{Ye2014b}.

The point of the Orlicz-Brunn-Minkowski theory and its dual is, of course, that they extend the $L_p$-Brunn-Minkowski theory and its dual, in a nontrivial and productive fashion.  Thus when $\varphi(x_1,x_2)=x_1^p+x_2^p$ and $\varphi(t)=t^p$, the two inequalities stated above become the corresponding dual $L_p$-Brunn-Minkowski and dual $L_p$-Minkowski inequality, respectively.  Other choices are possible, however.  It is particularly interesting to note that when $\varphi(t)=\log t$, the dual Orlicz-Minkowski inequality becomes the precise dual of the conjectured (and so far proved only for $n=2$) $\log$-Minkowski inequality (see \cite[p.~1976]{BLYZ} and Theorem~\ref{starlog} below and the remarks thereafter).  Moreover, a suitable choice for $\varphi(x_1,x_2)$ in the dual Orlicz Brunn-Minkowski inequality yields the precise dual of the $\log$-Brunn-Minkowski inequality, proved in \cite{BLYZ} to be equivalent to the $\log$-Minkowski inequality (see \cite[Problem~1.1]{BLYZ} and Section~\ref{DBMI} below).

Section~\ref{radialMadd} is dedicated to a new addition of star sets that we call radial $M$-addition.  Once again this is dual to a concept in the classical Brunn-Minkowski theory, called $M$-addition.  Briefly, the $M$-sum (or radial $M$-sum) of two sets is a very natural generalization of Minkowski addition (or radial addition, respectively) in which coefficients for linear combinations (or radial linear combinations, respectively) are taken from the coordinates of vectors in a set $M\subset\R^2$.  Introduced in a special situation by Protasov \cite{Pro97, Pro99}, $M$-addition was rediscovered, generalized, and systematically investigated in \cite{GHW}, where it was shown that any continuous and $GL(n)$-covariant operation between origin-symmetric compact convex sets must be an $M$-addition for some compact convex $M$ symmetric with respect to the coordinate axes.  The significance of $M$-addition was heightened when in \cite{GHW2} it was proved that in this context (addition of origin-symmetric compact convex sets), $M$-addition and Orlicz addition are essentially equivalent.  In Theorem~\ref{orlmadd}, we prove that this is also true for radial $M$-addition and radial Orlicz addition if $\varphi\in \Phi_m$ is a convex function, but otherwise examples show that there is generally not such a close relationship between the two additions.

Special cases of some of the results in this paper were obtained independently in \cite[Chapter~5]{ZZ} (see also \cite{ZZX}).  A brief description of the genesis of our paper and how it compares to \cite{ZZ} may be found in the Appendix.

\section{Definitions and preliminaries}\label{prelim}
Mostly we follow \cite[Section 2]{GHW2} and in an effort to keep this paper short we refer the reader there for all unexplained notation and terminology, which in any case is rather standard for convex geometry.  In this section we therefore focus on new ingredients, i.e., those not used in \cite{GHW2}, and the few notations we adopt in the present paper different from those in \cite{GHW2}.

We denote by $o$ the origin in $\R^n$ and by $\{e_1,\dots,e_n\}$ its standard basis.  The closed unit ball in $\R^n$ is denoted by $B^n$.

We write $V_k$ for $k$-dimensional Hausdorff measure in $\R^n$, where $k\in \{1,\dots,n\}$.  The notation $dz$ means $dV_k(z)$ for the appropriate $k=1,\dots,n$, unless stated otherwise.

A set $K$ in $\R^n$ is {\it star-shaped at $o$} if $o\in K$ and for each   $x\in\R^n\setminus\{o\}$, the intersection $K\cap\{cx:c\ge 0\}$ is a (possibly degenerate) compact line segment. If $K$ is star-shaped at $o$, we define its {\it radial
function} $\rho_K$ for $x\in \R^n\setminus\{o\}$ by
$$
\rho_{K}(x)= \max\{c\ge 0:cx\in K\}.
$$
This definition is a slight modification of \cite[(0.28)]{Gar06};
as defined here, the domain of $\rho_K$ is always $\R^n\setminus\{o\}$.  Radial functions are {\em homogeneous of degree $-$1}, that is,
$$\rho_K(rx)=r^{-1}\rho_K(x),$$
for all $x\in \R^n\setminus\{o\}$ and $r>0$, and are therefore often regarded as functions on the unit sphere $S^{n-1}$.  Conversely, any nonnegative and homogeneous of degree $-1$ function on $\R^n\setminus\{o\}$ is the radial function of a unique subset of $\R^n$ that is star-shaped at $o$.

In this paper, a {\it star set} in $\R^n$ is a bounded Borel set that is star-shaped at $o$.   We denote the class of star sets in
$\R^n$ by ${\mathcal S}^n$.  Note that ${\mathcal S}^n$ is closed under finite unions, countable intersections, and intersections with subspaces. Also, if a set $K$ in $\R^n$ is star-shaped at $o$, then $K\in {\cS}^n$ if and only if $\rho_K$, restricted to $S^{n-1}$, is a bounded Borel-measurable function. Let $\mathcal{S}^n_c$ denote the class of {\em star bodies} in $\R^n$, i.e., star sets with a continuous radial function.  We write ${\mathcal S}^n_+$ and ${\mathcal S}^n_{c+}$ for the subclasses of ${\mathcal S}^n$ and ${\mathcal S}^n_c$, respectively, whose members have radial functions that are positive on $S^{n-1}$.  Then ${\mathcal S}^n_{c+}$ consists of star bodies that contain the origin in their interiors.  An extra subscript $s$ stands for origin-symmetric sets.  Our definitions and notation differ from those used elsewhere, such as \cite[Section~0.7]{Gar06}, \cite{GHW}, and \cite{GV}.

Define the {\em radial sum} $x\widetilde{+}y$ of $x,y\in \R^n$ by
$$
x\widetilde{+}y=\left\{\begin{array}{ll} x+y &
\mbox{if $x$, $y$, and $o$ are collinear,}\\
o & \mbox{otherwise.}
\end{array}\right.
$$
Then the radial linear combination $\alpha K\widetilde{+}\beta L$, where $K, L\in {\cS}^n$ and $\alpha, \beta\ge 0$, can be defined either by
$$
\alpha K\widetilde{+}\beta L=\{\alpha x\widetilde{+}\beta y: x\in K, y\in L\},
$$
or by
\begin{equation}\label{radials2}
\rho_{\alpha K\widetilde{+}\beta L}=\alpha\rho_K+\beta\rho_L.
\end{equation}
More generally, for $p\in \R$, $p\neq 0$, the {\em $p$th radial linear combination} $\alpha K\widetilde{+}_p\,\beta L$, where $K, L\in {\cS}^n$ and $\alpha, \beta\ge 0$, can be defined by
\begin{equation}\label{radialslp}
\rho_{\alpha K\widetilde{+}_p\,\beta L}^p=\alpha\rho_K^p+\beta\rho_L^p.
\end{equation}
Here (\ref{radialslp}) is interpreted to mean that if $p<0$ and $\rho_K(x)\rho_L(x)=0$, then $\rho_{\alpha K\widetilde{+}_p\,\beta L}(x)=0$.  See \cite[Section~5.4]{GHW}.  Clearly, $\alpha K\widetilde{+}_p\,\beta L\in {\cS}^n$.  The operations of {\em radial addition} and {\em $p$th radial addition} are the special cases of (\ref{radials2}) and (\ref{radialslp}), respectively, when $\alpha=\beta=1$.

The {\em radial metric} $\widetilde{\delta}$ defines the distance between star sets $K,L\in {\mathcal S}^n$ by
$$\widetilde{\delta}(K,L)=\|\rho_K-\rho_L\|_{\infty}=\sup_{u\in S^{n-1}}|\rho_K(u)-\rho_L(u)|.$$
The radial metric differs considerably from the Hausdorff metric; for example, the radial distance between any two different origin-symmetric line segments containing the origin and of length two is one.  Unless specified otherwise, all statements involving a topology on $({\mathcal S}^n)^m$, $m\in \N$, refer to that generated by $\widetilde{\delta}$.

The {\em dual cone measure} of a star set $K$ in $\R^n$ such that $V_n(K)>0$ is the Borel probability measure $\widetilde{V}_K$ in $S^{n-1}$ defined by
\begin{equation}\label{dualcone}
d\widetilde{V}_K(u)=\frac{\rho_{K}(u)^n}{nV_n(K)}\,du.
\end{equation}

Let $I$ be a possibly infinite interval in $\R$.  The {\em left derivative} and {\em right derivative} of a function $f:I\to\R$ are denoted by $f'_l$ and $f'_r$, respectively.

Let $\Phi_m$, $m\in \N$, be the set of all continuous functions
$\varphi:[0,\infty)^m\to [0,\infty)$ that are strictly increasing in each component and such that $\varphi(o)= 0$ and $\lim_{t\to \infty} \varphi(tx) =\infty$, for each $x\in [0,\infty)^m\setminus\{o\}$.  Let $\Psi_m$, $m\in \N$, be the set of all continuous functions $\varphi:(0,\infty)^m\to(0,\infty)$ that are strictly decreasing in each component and such that $\lim_{t\to 0} \varphi(tx) =\infty$ and $\lim_{t\to \infty} \varphi(tx) =0$, for each $x\in (0,\infty)^m$.  We also denote by $\Phi_m^{(1)}$ and $\Psi_1^{(1)}$ the classes of functions in $\Phi_m$ and $\Psi_1$, respectively, such that $\varphi(e_j)=1$ for $j=1,\dots,m$.  We caution the reader that similar notation was used in \cite{GHW2} for different classes of functions.

The prototype function is $\varphi(x_1,\dots,x_m)=x_1^p+\cdots+x_m^p$, which belongs to $\Phi_m^{(1)}$ if $p>0$ and to $\Psi_m$ if $p<0$.

\begin{rem}\label{rem1}
{\em These classes of functions are chosen for convenience. Several of the results below hold for more general classes of functions; for example, everything in Sections~\ref{Onorms} and~\ref{DBMI} holds when the limits 0 and $\infty$ in the definitions of $\Phi_m$ and $\Psi_m$ are replaced by limits contained in $[0,1)$ and $(1,\infty]$, respectively, provided the measure $\mu$ there is a probability measure.  The same applies, with appropriate modifications, when $\Phi_m$ is replaced by the class $\Phi_m'$ of all continuous functions $\varphi:(0,\infty)^m\to(0,\infty)$ that are strictly increasing in each component and such that $\lim_{t\to 0} \varphi(tx) <1$ and $\lim_{t \to \infty} \varphi(tx)>1$, for each $x\in (0,\infty)^m$, but in this case the star sets involved should have positive radial functions. Note that if $\varphi(x_1,\dots,x_m)=x_1\cdots x_m$, for example, then $\varphi$ is in $\Phi_m'$ if $\varphi$ is restricted to $(0,\infty)^m$, but $\varphi$ is not
in $\Phi_m$ as a function on $[0,\infty)^m$.}
\end{rem}

Jensen's inequality has many versions; see, for example, \cite[Lemma~1, p.~76 and Exercise~9, p.~80]{Fer}.  For the reader's convenience, we state the precise form we need and supply a brief proof.

\begin{prp}\label{prpJ}
{\rm{(Jensen's inequality.)}} Let $\mu$ be a probability measure in a space $X$, let $U$ be an open convex set in $\R^n$, and let $\varphi$ be a convex real-valued function on $U$. Assume that $g:X\to U$ is measurable and component-wise $\mu$-integrable, and that $\varphi\circ g$ is $\mu$-integrable.
Let $z_0=\int_X g(x)\,d\mu(x)$.  Then $z_0\in U$ and
\begin{equation}\label{JenIneq}
\int_X\varphi(g(x))\,d\mu(x)\ge \varphi\left(\int_Xg(x)\,d\mu(x)\right).
\end{equation}
If $\varphi$ is strictly convex, then equality holds if and only if $g(x)=z_0$ for $\mu$-almost all $x\in X$.

If $\varphi$ is concave, then the inequality in \eqref{JenIneq} is reversed, with the same equality condition if $\varphi$ is strictly concave.
\end{prp}

\begin{proof}
Suppose that $\varphi$ is convex.  The fact that $z_0\in U$ follows from a separation argument. If $v$ belongs to the subgradient at $z_0$, which is nonempty since $U$ is open and $z_0\in U$, then
$$
\varphi(z)\ge \varphi(z_0)+\langle v,z-z_0\rangle,
$$
for all $z\in U$. If $x\in X$ and $g(x)=z$, we get
$$
\varphi(g(x))\ge \varphi(z_0)+\langle v,g(x)-z_0\rangle.
$$
Integration with respect to $\mu$, using the integrability assumptions and the definition of $z_0$, yields
$$
\int_X\varphi(g(x))\,d\mu(x)\ge \varphi(z_0)+\int_X\langle v,g(x)-z_0\rangle\,d\mu(x)= \varphi(z_0)+\langle v,o\rangle=\varphi(z_0).
$$
This proves (\ref{JenIneq}).  If equality holds in (\ref{JenIneq}) and $\varphi$ is strictly convex, then the previous display shows that we must have
$$
\varphi(g(x))=\varphi(z_0)+\langle v,g(x)-z_0\rangle,
$$
for $\mu$-almost all $x\in X$. Hence $g(x)=z_0$ for $\mu$-almost all $x\in X$.

If $\varphi$ is concave on $U$, the result follows from the above argument applied to the convex function $-\varphi$.
\end{proof}

\section{Radial Orlicz addition and Orlicz intersection bodies}\label{Onorms}
Let $m,n\ge 2$ and let $\mu$ be a nonzero finite Borel measure in $({\mathcal{S}}^n)^m$ with support contained in a bounded separable subset $C\subset({\mathcal{S}}^n)^m$ (with respect to the product radial metric). Note that $({\mathcal{S}}^n_c)^m$ is a separable subset, while $({\mathcal{S}}^n)^m$ itself is not separable.  For each $\varphi\in \Phi_m$, we define
\begin{equation}
\label{newdeff1}
\rho_{\Spm}(x)
=\inf\left\{\lambda>0: \int_{\left({\cS}^n\right)^m}\varphi\left(\frac{\rho_{K_1}(x)}{\lambda}
,\dots,\frac{\rho_{K_m}(x)}{\lambda}\right)\, d\mu(K_1,\dots,K_m)\le 1\right\},
\end{equation}
for all $x\in \R^n\setminus\{o\}$.  If $\varphi\in \Psi_m$, we assume in addition that the support of $\mu$ is contained in $(\mathcal{S}^n_+)^m$, and for $x\in \R^n\setminus\{o\}$, define $\rho_{\Spm}(x)$ by (\ref{newdeff1}), but with $\ge 1$ instead of $\le 1$.

This definition requires some discussion.   By our assumptions, there is an $M>0$ such that if $(K_1,\dots,K_m)\in C$, then $\rho_{K_j}(u)\le M$, for all $u\in S^{n-1}$ and $j=1,\dots,m$ (all sets are contained in $MB^n$). Let $u\in S^{n-1}$ and let $\lambda>0$. There is a unique $\tau>0$ such that $\varphi(\tau,\dots,\tau)=1/\mu(C)$. Then for $\varphi\in \Phi_m$, we have
\begin{equation}\label{incinc}
\int_C\varphi\left(\frac{\rho_{K_1}(u)}{\lambda}
,\dots,\frac{\rho_{K_m}(u)}{\lambda}\right)\, d\mu(K_1,\dots,K_m)
\le \mu(C)\varphi(M/\lambda,\dots,M/\lambda)\le 1
\end{equation}
provided $\lambda\ge M/\tau$. Therefore $\rho_{\Spm}(u)\le M/\tau$ and hence $\rho_{\Spm}$ is bounded.  For $\varphi\in \Psi_m$, the inequalities in (\ref{incinc}) are reversed, but in view of the reversed inequality in (\ref{newdeff1}) and the fact that $\varphi$ is decreasing in each component, the conclusion is the same. Since the function $\rho_{\Spm}$ on $\R^n\setminus\{o\}$ just defined is nonnegative and homogeneous of degree $-1$, it is the radial function of a set that is star-shaped at $o$. Next, observe that the function on $C\times S^{n-1}$ that maps $(K_1,\dots,K_m,u)$ to the integrand in (\ref{newdeff1}) is continuous in each of the first $m$ variables and Borel measurable in $u$.  It follows that it is jointly Borel measurable.  Here we use the fact that $C$ is separable and \cite[Exercise~11.3]{Kec}, which can be solved by adjusting the argument in \cite[Theorem~1]{Rud}.  Therefore $\rho_{\Spm}$ is Borel measurable and thus $S_{\varphi,\mu}\in \mathcal{S}^n$.

In particular, in the very special but important case when $\mu$ is defined by (\ref{mugood}) below, $S_{\varphi,\mu}$ is always a star set.

An alternative description of $\rho_{\Spm}(x)$, $x\in\R^n\setminus\{o\}$, is as follows. We first consider the case $\varphi\in\Phi_m$. If $\rho_{\Spm}(x)>0$, then by Lebesgue's dominated convergence theorem,  $\rho_{\Spm}(x)$ is the unique $\lambda=\lambda(x)>0$ such that
\begin{equation}\label{eqeq1}
\int_{\left({\cS}^n\right)^m}\varphi\left(\frac{\rho_{K_1}(x)}{\lambda}
,\dots,\frac{\rho_{K_m}(x)}{\lambda}\right)\, d\mu(K_1,\dots,K_m)= 1.
\end{equation}
If  $\rho_{\Spm}(x)=0$, then $\mu$ is concentrated on the set of all $(K_1,\ldots,K_m)\in (\mathcal{S}^n)^m$ satisfying $\rho_{K_j}(x)=0$, for $j=1,\ldots,m$.
If also $C\subset (\mathcal{S}^n_c)^m$, then the function on $C\times S^{n-1}$ that maps $(K_1,\dots,K_m,u)$ to the integrand in (\ref{newdeff1}) is  continuous in $u$. Lebesgue's dominated convergence theorem then yields the continuity of $\rho_{\Spm}$ so that in this case, we have $S_{\varphi,\mu}\in \mathcal{S}^n_c$. Next we consider the case $\varphi\in\Psi_m$. In order to obtain \eqref{eqeq1} again, we make the additional assumption that there is some fixed
$L\in \mathcal{S}^n_+$ such that $L\subset K_j$, for $j=1,\ldots,m$, whenever $(K_1,\ldots,K_m)\in C$. Then, for each $u\in S^{n-1}$ and $\lambda>0$,
\begin{equation}\label{incinc2}
\int_C\varphi\left(\frac{\rho_{K_1}(u)}{\lambda}
,\dots,\frac{\rho_{K_m}(u)}{\lambda}\right)\, d\mu(K_1,\dots,K_m)
\le \mu(C)\varphi(\rho_{L}(u)/\lambda,\dots,\rho_{L}(u)/\lambda)<\infty.
\end{equation}
For $0<\lambda\le \rho_{L}(u)/\tau$, the right side of \eqref{incinc2} is bounded from above by $1$. Now we can argue as before to see that  $\rho_{\Spm}(x)$, $x\in\R^n\setminus\{o\}$, is the unique $\lambda=\lambda(x)>0$ such that \eqref{eqeq1} is satisfied, and hence that $S_{\varphi,\mu}\in\mathcal{S}^n_+$. Moreover, if $C\subset (\mathcal{S}^n_{c+})^m$ and if $L$ contains $rB^n$ for some $r>0$, then $S_{\varphi,\mu}\in\mathcal{S}^n_{c+}$.

\begin{lem}\label{dualconstlem}
Let $m,n\ge 2$, let $\varphi\in \Phi_m$, and let $\mu$ be a nonzero finite Borel measure in $({\cS}^n)^m$ with support contained in a bounded separable subset of $({\cS}^n)^m$.  If $A\in GL(n)$, then
$$
A\left(\Spm\right)=\SpAm.
$$
The same statement holds when $\varphi\in \Psi_m$ and ${\cS}^n$ is replaced by $\mathcal{S}^n_+$.
\end{lem}

\begin{proof}
We omit the details, since the proof is similar to that of \cite[Lemma~4.4(ii)]{GHW2}.  One replaces support functions of compact convex sets by radial functions of star sets and uses the formula $\rho_{AK}(x)=\rho_K(A^{-1}x)$ for $A\in GL(n)$ (see \cite[(0.33), p.~20]{Gar06}) for the change in a radial function under a nonsingular linear transformation, instead of the corresponding formula for the change in a support function.
\end{proof}

Let $m\ge 2$, let $\varphi\in \Phi_m$ (or $\varphi\in\Psi_m$), and for $j=1,\dots,m$, let $K_j\in {{\cS}}^n$ (or $K_j\in \mathcal{S}^n_+$, respectively). Define a measure $\mu$ in $({{\cS}}^n)^m$ (or $({{\cS}}^n_+)^m$, respectively) by
\begin{equation}\label{mugood}
\mu=\delta_{(K_1,\dots,K_m)}=\delta_{K_1}\times\cdots\times\delta_{K_m},
\end{equation}
where $\delta_x$ denotes the Dirac measure (a unit mass at $x$).  The corresponding {\em radial Orlicz sum} of $K_1,\dots,K_m$ is defined to be $\Spm$, where $\Spm$ is as in (\ref{newdeff1}), and is denoted by $\widetilde{+}_{\varphi}(K_1,\dots,K_m)$.  This means that for $\varphi\in \Phi_m$ and $K_j\in {{\cS}}^n$, $j=1,\dots,m$,
\begin{equation}\label{OrlComb1}
\rho_{\widetilde{+}_{\varphi}(K_1,\dots,K_m)}(x)=\inf\left\{\lambda>0: \varphi\left(\frac{\rho_{K_1}(x)}{\lambda}
,\dots,\frac{\rho_{K_m}(x)}{\lambda}\right)\le 1\right\},
\end{equation}
for all $x\in \R^n\setminus\{o\}$. Moreover, from our earlier remarks, or from (\ref{OrlComb1}) directly, it is clear that $\rho_{\widetilde{+}_{\varphi}(K_1,\dots,K_m)}$ is Borel measurable on $S^{n-1}$ and it follows that $\widetilde{+}_{\varphi}(K_1,\dots,K_m)\in {{\cS}}^n$.
Similarly, for $\varphi\in \Psi_m$ and $K_j\in \mathcal{S}^n_+$, $j=1,\dots,m$, $\rho_{\widetilde{+}_{\varphi}(K_1,\dots,K_m)}(x)$ is as in (\ref{OrlComb1}), but with $\ge 1$ instead of $\le 1$, and then  $\widetilde{+}_{\varphi}(K_1,\dots,K_m)\in {{\cS}}^n_+$.

Equivalently, for $\varphi\in \Phi_m$ and $K_j\in {{\cS}}^n$, $j=1,\dots,m$,  the radial Orlicz sum $\widetilde{+}_{\varphi}(K_1,\dots,K_m)$ can be defined implicitly (and uniquely) by
\begin{equation}\label{Orldef}
\varphi\left( \frac{\rho_{K_1}(x)}{\rho_{\widetilde{+}_{\varphi}(K_1,\dots,K_m)}(x)},\dots, \frac{\rho_{K_m}(x)}{\rho_{\widetilde{+}_{\varphi}(K_1,\dots,K_m)}(x)}\right) = 1,
\end{equation}
if $\rho_{K_1}(x)+\cdots+\rho_{K_m}(x)> 0$ and by $\rho_{\widetilde{+}_{\varphi}(K_1,\dots,K_m)}(x)=0$, otherwise, for all $x\in \R^n\setminus\{o\}$. For $\varphi\in \Psi_m$ and $K_j\in {{\cS}}^n_+$, $j=1,\dots,m$, the radial Orlicz sum $\widetilde{+}_{\varphi}(K_1,\dots,K_m)$ can also be defined implicitly (and uniquely) by (\ref{Orldef}). Here the set $L=K_1\cap\cdots\cap K_m\in\mathcal{S}^n_+$ can serve as the star set $L$ required before \eqref{incinc2}.

Note that if $\varphi\in {\Phi}_m$, then $\rho_{\widetilde{+}_{\varphi}(K_1,\dots,K_m)}(x)=0$ implies that $\rho_{K_1}(x)=\cdots=\rho_{K_m}(x)=0$. Also, if $\varphi\in {\Phi}_m^{(1)}$ and $\rho_{K_j}(x)=0$ for all $j\neq j_0$, then \eqref{Orldef} yields $\rho_{\widetilde{+}_{\varphi}(K_1,\dots,K_m)}(x)=\rho_{K_{j_0}}(x)$.

A consequence of our earlier remarks, or of (\ref{Orldef}) directly, is that $\widetilde{+}_{\varphi}:({\cS}^n_c)^m\rightarrow {\cS}^n_c$ for $\varphi\in \Phi_m$ and $\widetilde{+}_{\varphi}:({\cS}^n_{c+})^m\rightarrow {\cS}^n_{c+}$  for $\varphi\in \Phi_m\cup \Psi_m$.

An important special case is obtained when
\begin{equation}\label{spphi}
\varphi(x_1,\dots,x_m)=\sum_{j=1}^m\varphi_j(x_j),
\end{equation}
for some fixed $\varphi_j\in \Phi_1$, $j=1,\dots,m$ (or $\varphi_j\in \Psi_1$, $j=1,\dots,m$).  We then write
$$\widetilde{+}_{\varphi}(K_1,\dots,K_m)=K_1\widetilde{+}_{\varphi}
\cdots\widetilde{+}_{\varphi}K_m.$$
This means that $K_1\widetilde{+}_{\varphi}\cdots\widetilde{+}_{\varphi}K_m$ can be defined for all $x\in \R^n\setminus\{o\}$ and $\varphi_j\in \Phi_1$, $j=1,\dots,m$, by the corresponding special case
\begin{equation}\label{lastone}
\sum_{j=1}^m\varphi_j\left( \frac{\rho_{K_j}(x)}{\rho_{K_1\widetilde{+}_{\varphi}
\cdots\widetilde{+}_{\varphi}K_m}(x)}\right) = 1
\end{equation}
of (\ref{Orldef}), when $\rho_{K_1}(x)+\cdots +\rho_{K_m}(x)>0$, and by $\rho_{K_1\widetilde{+}_{\varphi}\cdots\widetilde{+}_{\varphi}K_m}(x)=0$, otherwise, and similarly by (\ref{lastone}) when $\varphi_j\in \Psi_1$, $j=1,\dots,m$.

\begin{thm}\label{dualOrthm1}
If $m,n\ge 2$ and $\varphi\in \Phi_m$, then radial Orlicz addition $\widetilde{+}_{\varphi}:({\cS}^n)^m\rightarrow {\cS}^n$

\noindent{\rm(i)} is $GL(n)$ covariant, i.e., $A(\widetilde{+}_{\varphi}(K_1,\dots,K_m))=\widetilde{+}_{\varphi}(AK_1,\dots,AK_m)$ for $A\in GL(n)$ and $K_1,\dots,K_m\in {\cS}^n$;

\noindent{\rm(ii)} satisfies $+_{\varphi}:({\cS}^n_s)^m\rightarrow {\cS}^n_s$;

\noindent{\rm(iii)} is homogeneous of degree 1, i.e., $\widetilde{+}_{\varphi}(rK_1,\dots,rK_m)=r\widetilde{+}_{\varphi}(K_1,\dots,K_m)$
for $r\ge 0$ and $K_1,\dots,K_m\in {\cS}^n$;

\noindent{\rm(iv)} is section covariant, i.e., $\widetilde{+}_{\varphi}(K_1,\dots,K_m)\cap S=\widetilde{+}_{\varphi}(K_1\cap S,\dots,K_m\cap S)$ for any subspace $S$ of $\R^n$ and $K_1,\dots,K_m\in {\cS}^n$;

\noindent{\rm(v)} has the identity property, i.e., $\widetilde{+}_{\varphi}(\{o\},\dots,\{o\},K_j,\{o\},\dots,\{o\})=K_j$ for $j=1,\dots,m$ and $K_j\in {\cS}^n$, provided $\varphi\in {\Phi}_m^{(1)}$;

\noindent{\rm(vi)} is monotonic, i.e., if $K_j\subset L_j$ for $K_j,L_j\in {\cS}^n$, $j=1,\dots,m$, then $\widetilde{+}_{\varphi}(K_1,\dots,K_m)\subset \widetilde{+}_{\varphi}(L_1,\dots,L_m)$;

\noindent{\rm(vii)} is such that $\widetilde{+}_{\varphi}:({\cS}^n)^m\rightarrow {\cS}^n$ is continuous in the sense of pointwise convergence of radial functions and $\widetilde{+}_{\varphi}:({\cS}^n_{c+})^m\rightarrow {\cS}^n_{c+}$ is continuous in the radial metric.

With $r>0$ in {\rm(iii)}, all statements except {\rm(v)} hold when $\varphi\in \Psi_m$ and ${\cS}^n$ is replaced by ${\cS}^n_+$.
\end{thm}

\begin{proof}  (i) This follows from Lemma~\ref{dualconstlem} in the same way that the $GL(n)$ covariance of Orlicz addition given in \cite[Theorem~5.2]{GHW2} follows from \cite[Lemma~4.4(ii)]{GHW2}.

Parts (ii) and (iii) are direct consequences of (i) applied to the maps $Ax=-x$ and $Ax=rx$, $x\in \R^n$, respectively.

Parts (iv) and (v) follow easily from definition (\ref{Orldef}) and the remarks thereafter; for (iv), one uses the obvious facts that $\rho_{K\cap S}(x)=\rho_K(x)$ if $x\in S$ and $\rho_{K\cap S}(x)=0$ if $x\not\in S$, for $K\in\cS^n$ and $x\in \R^n\setminus \{o\}$.

(vi) This also follows easily from (\ref{Orldef}); the proof is the same as that of the monotonicity of Orlicz addition in \cite[Theorem~5.2]{GHW2}, on replacing support functions by radial functions.

(vii) Let $K_{ij}\in \cS^n$, $i\in \N\cup\{0\}$, $j=1,\dots,m$, be such that $\rho_{K_{ij}}(u)\rightarrow \rho_{K_{0j}}(u)$ for all $u\in S^{n-1}$ as $i\rightarrow\infty$.  The desired conclusion that $\rho_{\widetilde{+}_{\varphi}(K_{i1},\dots,K_{im})}\rightarrow
\rho_{\widetilde{+}_{\varphi}(K_{01},\dots,K_{0m})}$ pointwise as $i\rightarrow\infty$ follows from the continuity of $\varphi$ and the uniqueness of the solution of \eqref{Orldef}.

Now let $K_{ij}\in \cS_{c+}^n$, $i\in \N\cup\{0\}$, $j=1,\dots,m$, be such that $\widetilde{\delta}\left(\rho_{K_{ij}}, \rho_{K_{0j}}\right)\rightarrow 0$, that is,  $\rho_{K_{ij}}\rightarrow \rho_{K_{0j}}$ uniformly on $S^{n-1}$ as $i\rightarrow\infty$, for $j=1,\dots,m$.  From the previous paragraph we know that $\rho_{\widetilde{+}_{\varphi}(K_{i1},\dots,K_{im})}\rightarrow
\rho_{\widetilde{+}_{\varphi}(K_{01},\dots,K_{0m})}$ pointwise as $i\rightarrow\infty$.  If the convergence is not uniform on $S^{n-1}$, then there is an $\ee>0$ such that for all $i\ge 1$, there are $n_i\ge i$ and $u_{n_i}\in S^{n-1}$ such that
\begin{equation}\label{dy1}
|\rho_{\widetilde{+}_{\varphi}(K_{n_i1},\dots,K_{n_im})}(u_{n_i})-
\rho_{\widetilde{+}_{\varphi}(K_{01},\dots,K_{0m})}(u_{n_i})|\ge \ee.
\end{equation}
Since $S^{n-1}$ is compact, we may assume that $u_{n_i}\to u_0$ as $i\to \infty$.  We claim that we may also assume that there is a $c_0>0$ such that $\rho_{\widetilde{+}_{\varphi}(K_{n_i1},\dots,K_{n_im})}(u_{n_i})\to c_0$ as $i\to\infty$.  To see this, note that since $\rho_{K_{ij}}\rightarrow \rho_{K_{0j}}$ uniformly on $S^{n-1}$, we have $c_1\le \rho_{K_{ij}}(u)\le c_2$, for some $c_1, c_2>0$ and all $u\in S^{n-1}$, $i\in \N\cup\{0\}$, and $j=1,\dots,m$.  From (\ref{Orldef}) and our assumptions on $\varphi$, it follows that if $\varphi\in \Phi_m$ and $\tau>0$ are such that $\varphi(\tau,\dots,\tau)=1$, then
\begin{equation}\label{ee1}
c_1/\tau\le \rho_{\widetilde{+}_{\varphi}(K_{n_i1},\dots,K_{n_im})}(u)\le c_2/\tau,
\end{equation}
for all $u\in S^{n-1}$ and $i\in \N\cup\{0\}$.  If $\varphi\in \Psi_m$, then (\ref{ee1}) holds with the inequalities reversed.  In either case, the sequence $(\rho_{\widetilde{+}_{\varphi}(K_{n_i1},\dots,K_{n_im})}(u_{n_i}))$ is bounded and thus has a convergent subsequence, proving the claim.

From (\ref{Orldef}) and the fact that $\rho_{K_{ij}}\rightarrow \rho_{K_{0j}}$ uniformly on $S^{n-1}$ as $i\rightarrow\infty$, for $j=1,\dots,m$, we now obtain
$$1=\varphi\left(\frac{\rho_{K_{n_i1}}(u_{n_i})}{\rho_{\widetilde{+}_{\varphi}
(K_{n_i1},\dots,K_{n_im})}(u_{n_i})},\dots,\frac{\rho_{K_{n_im}}(u_{n_i})}
{\rho_{\widetilde{+}_{\varphi}(K_{n_i1},\dots,K_{n_im})}(u_{n_i})}\right)\to
\varphi\left(\frac{\rho_{K_{01}}(u_{0})}{c_0},\dots,\frac{\rho_{K_{0m}}(u_{0})}
{c_0}\right),$$
as $i\to\infty$, and hence by the previous expression and (\ref{Orldef}) again, we have $c_0=\rho_{\widetilde{+}_{\varphi}(K_{01},\dots,K_{0m})}(u_{0})$.  But (\ref{dy1}) implies that
$$|\rho_{\widetilde{+}_{\varphi}(K_{n_i1},\dots,K_{n_im})}(u_{n_i})-
\rho_{\widetilde{+}_{\varphi}(K_{01},\dots,K_{0m})}(u_{n_i})|\to
|c_0-\rho_{\widetilde{+}_{\varphi}(K_{01},\dots,K_{0m})}(u_{0})|\ge\ee,$$
as $i\to \infty$, a contradiction.  Therefore $\rho_{\widetilde{+}_{\varphi}(K_{i1},\dots,K_{im})}\rightarrow
\rho_{\widetilde{+}_{\varphi}(K_{01},\dots,K_{0m})}$ uniformly on $S^{n-1}$ as $i\rightarrow\infty$, which is equivalent to the convergence of the corresponding star bodies in the radial metric.
\end{proof}

For $\varphi\in\Phi_2$, radial Orlicz addition $\widetilde{+}_{\varphi}:({\cS}^n)^2\rightarrow {\cS}^n$ is clearly commutative if and only if $\varphi(x_1,x_2)=\varphi(x_2,x_1)$ for all $x_1,x_2\ge 0$, and the corresponding statement is true for $\varphi\in\Psi_2$.  We also have the following result.

\begin{cor}\label{nonassoc}
Let $n\ge 2$.  Radial Orlicz addition $\widetilde{+}_{\varphi}:({\cS}^n)^2\rightarrow {\cS}^n$, for $\varphi\in \Phi_2$, and $\widetilde{+}_{\varphi}:({\cS}^n_+)^2\rightarrow {\cS}^n_+$, for $\varphi\in \Psi_2$, is associative if and only if it is $p$th radial addition for some $p\in \R$, $p\neq 0$.
\end{cor}

\begin{proof}
By Theorem~\ref{dualOrthm1}, radial Orlicz addition is continuous in the sense of pointwise convergence of radial functions, homogeneous of degree 1, $GL(n)$ covariant, and section covariant.  Let $\varphi\in \Phi_2$. By the proof of \cite[Theorem~7.17]{GHW}, the restriction $\widetilde{+}_{\varphi}:({\cS}^n_s)^2\rightarrow {\cS}^n_s$ to the $o$-symmetric sets is either $p$th radial addition for some $p\in \R$, $p\neq 0$, or we have $\widetilde{+}_{\varphi}(K,L)=\{o\}$, or $\widetilde{+}_{\varphi}(K,L)=K$, or $\widetilde{+}_{\varphi}(K,L)=L$, for all $K,L\in {\cS}^n_{s}$.  (Note that while the continuity in \cite[Theorem~7.17]{GHW} is with respect to the radial metric, the proof only requires continuity in the sense of pointwise convergence of radial functions.) However, since $\varphi\in \Phi_2$, the latter three possibilities are excluded and in fact we must have $p>0$.

Now let $K,L\in {\cS}^n$, let $x\in \R^n\setminus\{o\}$, and let $\rho_K(x)=a$ and $\rho_L(x)=b$.  Choose $K',L'\in {\cS}^n_{s}$ such that $\rho_{K'}(x)=a\ge 0$ and $\rho_{L'}(x)=b\ge 0$.  Then by (\ref{Orldef}),  both $\rho_{\widetilde{+}_{\varphi}(K,L)}(x)$ and $\rho_{\widetilde{+}_{\varphi}(K',L')}(x)$ equal the unique $\lambda>0$ such that $\varphi(a/\lambda,b/\lambda)=1$.  Therefore,
$$\rho_{\widetilde{+}_{\varphi}(K,L)}(x)^p=
\rho_{\widetilde{+}_{\varphi}(K',L')}(x)^p=a^p+b^p=
\rho_{K}(x)^p+\rho_{L}(x)^p,$$
which completes the proof when $\varphi\in \Phi_2$.  Essentially the same proof works for $\varphi\in \Psi_2$ when ${\cS}^n$ is replaced by ${\cS}^n_+$, with same conclusion but with $p<0$. The required changes concern the function $f$ used in the proof of \cite[Theorem~7.17]{GHW}, which is now chosen as a function $f:(0,\infty)^2\to (0,\infty)$, and the use of \cite[Theorem~1]{Pearson} instead of \cite[Theorem~2]{Pearson}.
\end{proof}

There is also a natural definition of an {\em Orlicz intersection body $I_\varphi K$ of a star body $K\in {\cS}^n_c$}. By analogy with the $L_p$-intersection body of a star body (see \cite{Hab08}), when $\varphi\in \Phi_1$, we define $I_\varphi K$ to be the star body with radial function
\begin{equation}\label{intbody1}
\rho_{I_\varphi K}(u)=\inf\left\{\lambda>0: \int_{K}\varphi\left(\frac{1}{\lambda|u\cdot x|}\right)\, dx\le 1\right\},
\end{equation}
for $u\in S^{n-1}$, with suitable restrictions imposed on $\varphi$ so that the infimum exists.  More generally, when $\varphi\in \Phi_1$, an {\em Orlicz intersection body} is a star body $I_\varphi\mu$ whose radial function satisfies
\begin{equation}\label{intbody2}
\rho_{I_\varphi\mu}(u)=\inf\left\{\lambda>0: \int_{\R^n}\varphi\left(\frac{1}{\lambda|u\cdot x|}\right)\, d\mu(x)\le 1\right\},
\end{equation}
for $u\in S^{n-1}$, where $\mu$ is a finite Borel measure in $\R^n$.  When $\varphi\in \Psi_1$, the inequality $\le 1$ in (\ref{intbody1}) and (\ref{intbody2}) should be replaced by $\ge 1$. Variants of these definitions are conceivable.  We leave the investigation of these bodies for a future study.

\section{Dual Brunn-Minkowski-type inequalities}\label{DBMI}

The following result provides a dual Orlicz-Brunn-Minkowski inequality.

\begin{thm}\label{dualOBMI1}
Let $m, n\ge 2$, let $\varphi\in \Phi_m$, and let $\varphi_0(x_1,\dots,x_m)=\varphi(x_1^{1/n},\dots,x_m^{1/n})$ for $(x_1,\dots,x_m)\in [0,\infty)^m$.  If $\varphi_0$ is concave, then for all $K_j\in {\cS}^n$, $j=1,\dots,m$, with some $V_n(K_j)>0$,
\begin{equation}\label{obmi}
\varphi\left(\left(\frac{V_n(K_1)}{V_n(\widetilde{+}_{\varphi}(K_1,\dots,
K_m))}\right)^{1/n},\dots,
\left(\frac{V_n(K_m)}{V_n(\widetilde{+}_{\varphi}(K_1,\dots,
K_m))}\right)^{1/n}\right)\ge 1,
\end{equation}
while if $\varphi_0$ is convex, the inequality is reversed.  The same statements hold if instead $\varphi\in \Psi_m$ and $K_j\in {\cS}_{+}^n$, $j=1,\dots,m$.

If $\varphi_0$ is strictly concave (or convex, as appropriate) and $K_j\in {\cS}_{c+}^n$, $j=1,\dots,m$, equality holds if and only if $K_j$, $j=1,\dots,m$, are dilatates.
\end{thm}

\begin{proof}
Let $\varphi\in \Phi_m$, let $\varphi_0$ be concave, and initially, suppose that $K_j\in {\cS}^n_+$, $j=1,\dots,m$.  Then $\rho_{\widetilde{+}_{\varphi}(K_1,\dots,K_m)}(u)>0$, for $u\in S^{n-1}$. Hence,  $V_n(\widetilde{+}_{\varphi}(K_1,\dots,
K_m))>0$ and the dual cone measure $\widetilde{V}_{\widetilde{+}_{\varphi}(K_1,\dots,
K_m)}$ (see (\ref{dualcone}) with $K$ replaced by $\widetilde{+}_{\varphi}(K_1,\dots,
K_m)$) is a probability measure in $S^{n-1}$ with positive density with respect to $V_{n-1}$ in $S^{n-1}$.  We will use Jensen's inequality for concave functions (the reverse of (\ref{JenIneq})), with $X=S^{n-1}$, $\mu=\widetilde{V}_{\widetilde{+}_{\varphi}(K_1,\dots,
K_m)}$, $\varphi$ replaced by $\varphi_0$, $\R^n$ replaced by $\R^m$, $U=(0,\infty)^m$, and $g$ defined by
$$g(u)=\left(\frac{\rho_{K_1}(u)^n}{
\rho_{\widetilde{+}_{\varphi}(K_1,\dots,
K_m)}(u)^n},\dots,
\frac{\rho_{K_m}(u)^n}{\rho_{\widetilde{+}_{\varphi}(K_1,\dots,
K_m)}(u)^n}
\right).$$
With this and (\ref{Orldef}), we obtain
\begin{eqnarray*}
1&=& \int_{S^{n-1}}\varphi\left(\frac{\rho_{K_1}(u)}{
\rho_{\widetilde{+}_{\varphi}(K_1,\dots,
K_m)}(u)},\dots,
\frac{\rho_{K_m}(u)}{\rho_{\widetilde{+}_{\varphi}(K_1,\dots,
K_m)}(u)}
\right)\,d\widetilde{V}_{\widetilde{+}_{\varphi}(K_1,\dots,
K_m)}(u)\\
&=& \int_{S^{n-1}}\varphi_0\left(\frac{\rho_{K_1}(u)^n}{
\rho_{\widetilde{+}_{\varphi}(K_1,\dots,
K_m)}(u)^n},\dots,
\frac{\rho_{K_m}(u)^n}{\rho_{\widetilde{+}_{\varphi}(K_1,\dots,
K_m)}(u)^n}
\right)\,d\widetilde{V}_{\widetilde{+}_{\varphi}(K_1,\dots,
K_m)}(u)\\
&\le & \varphi_0\left(\int_{S^{n-1}}\frac{\rho_{K_1}(u)^n}{
\rho_{\widetilde{+}_{\varphi}(K_1,\dots,
K_m)}(u)^n}\,
d\widetilde{V}_{\widetilde{+}_{\varphi}(K_1,\dots,
K_m)}(u),\right.\\
& &
\hspace{1.5in}\left.\dots,\int_{S^{n-1}}\frac{\rho_{K_m}(u)^n}{
\rho_{\widetilde{+}_{\varphi}(K_1,\dots,
K_m)}(u)^n}\,
d\widetilde{V}_{\widetilde{+}_{\varphi}(K_1,\dots,
K_m)}(u)\right)\\
&=&\varphi_0\left(\frac{V_n(K_1)}{V_n(\widetilde{+}_{\varphi}(K_1,\dots,
K_m))},\dots,
\frac{V_n(K_m)}{V_n(\widetilde{+}_{\varphi}(K_1,\dots,
K_m))}\right)\\
&=&\varphi\left(\left(\frac{V_n(K_1)}{V_n(\widetilde{+}_{\varphi}(K_1,\dots,
K_m))}\right)^{1/n},\dots,
\left(\frac{V_n(K_m)}{V_n(\widetilde{+}_{\varphi}(K_1,\dots,
K_m))}\right)^{1/n}\right).
\end{eqnarray*}
Now suppose that $K_j\in {\cS}^n$, $j=1,\dots,m$, with $V_n(K_{j_0})>0$ for some $j_0$. Let $\ee>0$ and define $K_j(\varepsilon)\in\mathcal{S}^n_+$ by $\rho_{K_j(\varepsilon)}(u)=\rho_{K_j}+\varepsilon$, for $u\in S^{n-1}$ and $j=1,\ldots,m$. Then $\rho_{K_j(\varepsilon)}\downarrow\rho_{K_j}$ pointwise as $\varepsilon\downarrow 0$. By the Lebesgue dominated convergence theorem,
$V_n(K_j(\varepsilon))\downarrow V_n(K_j)$ as $\varepsilon\downarrow 0$, for $j=1,\dots,m$. Moreover,  $\rho_{\widetilde{+}_{\varphi}(K_1(\varepsilon),\dots,
K_m(\varepsilon))}\downarrow\rho_{\widetilde{+}_{\varphi}(K_1,\dots,K_m)}$ pointwise as $\varepsilon\downarrow 0$. By the Lebesgue dominated convergence theorem again,
we obtain
$$
V_n(\widetilde{+}_{\varphi}(K_1(\varepsilon),\dots,
K_m(\varepsilon))\downarrow V_n(\widetilde{+}_{\varphi}(K_1,\dots,
K_m))\ge V_n(\widetilde{+}_{\varphi}(\{o\},\dots,K_{j_0},\dots,
\{o\}))>0,
$$
since $\rho_{K_{j_0}}>0$ on a subset of $S^{n-1}$ of positive $V_{n-1}$-measure. Since $\varphi$ is continuous and (\ref{obmi}) holds with $K_j$ replaced by $K_j(\ee)$, $j=1,\dots,m$, the required conclusion follows by taking the limit as $\varepsilon\downarrow 0$.

Suppose that $K_j\in {\cS}_{c+}^n$, $j=1,\dots,m$, and that equality holds in \eqref{obmi}. Then equality also holds in Jensen's inequality.  If $\varphi_0$ is strictly concave, the equality condition for Jensen's inequality, together with the fact that the support of $\widetilde{V}_{\widetilde{+}_{\varphi}(K_1,\dots,
K_m)}$ is now $S^{n-1}$ and all radial functions involved are continuous, imply that $\rho_{K_j}(u)/\rho_{\widetilde{+}_{\varphi}(K_1,\dots,
K_m)}(u)$, $j=1,\dots,m$, and hence $\rho_{K_1}(u)/\rho_{K_j}(u)$, $j=1,\dots,m$, are constant on $S^{n-1}$. This establishes the equality condition.

The remainder of the theorem follows easily by similar arguments.
\end{proof}

Taking $m=2$, $K_1=K$, $K_2=L$, and $\varphi(x_1,x_2)=x_1^p+x_2^p$, $p\in \R$, $p\neq 0$, in Theorem~\ref{dualOBMI1}, we obtain the {\em dual $L_p$-Brunn-Minkowski inequality}
\begin{equation}\label{duallpBM}
V_n(K\widetilde{+}_p\,L)^{p/n}\le V_n(K)^{p/n}+V_n(L)^{p/n},
\end{equation}
and its equality conditions, where $p\in (0,n]$ and $\widetilde{+}_p$ is defined by (\ref{radialslp}), and where the inequality in (\ref{duallpBM}) is reversed if $p<0$ or $p>n$. See \cite[(85), p.~398]{Gar02}.

Of course, many other choices are possible.  For example, let $t\in (0,1)$ and let $\varphi_t(x_1,x_2)=x_1^{n(1-t)}x_2^{nt}$, for $x_1,x_2>0$.  Then $\varphi_t\not\in\Phi_2$, but $\varphi_t\in\Phi'_2$, where $\Phi'_2$ is as in Remark~\ref{rem1}. In this case radial Orlicz addition coincides with the radial log combination $(1-t)K\widetilde{+}_0\,tL$, i.e.,
$$\rho_{\widetilde{+}_{\varphi_t}(K, L)}(x)=\rho_{(1-t)K\widetilde{+}_0\,tL}(x)=
\rho_K(x)^{1-t}\rho_L(x)^t,$$
for $x\in \R^n\setminus\{o\}$ and $K, L\in \mathcal{S}_{+}^n$.  With this choice of $\varphi$, Theorem~\ref{dualOBMI1} yields
$$V_n((1-t)K\widetilde{+}_0\,tL)\le V_n(K)^{1-t}V_n(L)^t,$$
for all $K, L\in \mathcal{S}_{+}^n$.  This {\em dual log-Brunn-Minkowski inequality} is in fact the precise dual of the conjectured (and so far proved only for $n=2$) $\log$-Brunn-Minkowski inequality (see \cite[Problem~1.1]{BLYZ}):
$$V_n((1-t)K+_0\,tL)\ge V_n(K)^{1-t}V_n(L)^t,$$
where $(1-t)K+_0\,tL$ is the log Minkowski combination of origin-symmetric convex bodies $K$ and $L$ in $\R^n$.

\section{Radial Orlicz linear combination and dual Orlicz mixed volume}\label{radOlc}
Let $m,n\ge 2$ and suppose that $\alpha_j> 0$, $j=1,\dots,m$, and either $\varphi_j\in \Phi_1$, $j=1,\dots,m$, or $\varphi_j\in \Psi_1$, $j=1,\dots,m$.  Let
\begin{equation}\label{phrest}
\varphi(x_1,\dots,x_m)=\sum_{j=1}^m\alpha_j\varphi_j(x_j).
\end{equation}
We define the {\em radial Orlicz linear combination} $\widetilde{+}_{\varphi}(K_1,\dots,K_m,\alpha_1,\dots,\alpha_m)$ for all $x\in \R^n\setminus\{o\}$ and $\varphi_j\in \Phi_1$ and $K_j\in{\mathcal S}^n$, $j=1,\dots,m$, by taking the function $\varphi$ in (\ref{OrlComb1}) to be as in (\ref{phrest}); in other words, by
\begin{equation}\label{hK021}
\rho_{\widetilde{+}_{\varphi}(K_1,\dots,K_m,\alpha_1,\dots,\alpha_m)
}(x)=\inf\left\{\lambda>0:\sum_{j=1}^m\alpha_j\,
\varphi_j\left(\frac{\rho_{K_j}(x)}
{\lambda}\right)\le 1\right\}.
\end{equation}
We use the same definition (\ref{hK021}) for $\varphi_j\in \Psi_1$ and $K_j\in{\mathcal S}^n_+$, $j=1,\dots,m$, with $\le 1$ replaced by $\ge 1$.

For our purposes, it suffices to focus on the special case when $m=2$, $\alpha_1=1$, and $\alpha_2=\ee>0$. Henceforth we shall write  $K\widetilde{+}_{\varphi,\ee}L$ instead of $\widetilde{+}_{\varphi}(K,L,1,\ee)$.  The radial Orlicz linear combination $K\widetilde{+}_{\varphi,\ee}L$ can be defined implicitly (and uniquely) for $x\in \R^n\setminus\{o\}$ by
\begin{equation}\label{rlc}
\varphi_1\left(\frac{\rho_K(x)}{\rho_{K\widetilde{+}_{\varphi,\ee} L}(x)}\right)+\ee\varphi_2
\left(\frac{\rho_L(x)}{\rho_{K\widetilde{+}_{\varphi,\ee} L}(x)}\right)=1,
\end{equation}
if $\varphi_1,\varphi_2\in \Phi_1$ and $\rho_K(x)+\rho_L(x)>0$, or if $\varphi_1,\varphi_2\in \Psi_1$ and $K,L\in \mathcal{S}^n_+$.

Note that we have $K\widetilde{+}_{\varphi,\ee}L\in {\cS}^n_c$ (or $K\widetilde{+}_{\varphi,\ee}L\in {\cS}^n_{c+}$) if $K, L\in {\cS}^n_c$ (or $K, L\in {\cS}^n_{c+}$, respectively).

The following lemma is a dual analog of \cite[Lemma~8.2]{GHW2}.  It requires a different proof, since convergence of radial functions does not imply their uniform convergence, as it does for support functions.

\begin{lem}\label{newconvlem}
Suppose that either $\varphi_j\in {\Phi}_1^{(1)}$, $j=1,2$, $K,L\in {\mathcal{S}}^n$, and $c_1B^n\subset K$ for some $c_1>0$, or $\varphi_j\in {\Psi}_1^{(1)}$, $j=1,2$,  $K,L\in {\mathcal{S}}_{+}^n$, and $c_2B^n\subset L$ for some $c_2>0$.  Then $\rho_{K\widetilde{+}_{\varphi,\ee}L}\to \rho_{K}$ uniformly on $S^{n-1}$ as $\ee\to 0+$.
\end{lem}

\begin{proof}
Suppose that $\varphi_j\in {\Phi}_1^{(1)}$, $j=1,2$, and let $\ee\in (0,1]$.  Using (\ref{rlc}), both as it stands and with $\ee$ replaced by $1$, and the fact that $\varphi_1$ and $\varphi_2$ are strictly increasing, it is easy to see that $K\subset K\widetilde{+}_{\varphi,\ee}L\subset K\widetilde{+}_{\varphi,1}L$. Let $M_1<\infty$ be such that $L\subset M_1B^n$ and $K\widetilde{+}_{\varphi,1}L\subset M_1B^n$, and define
$$a_1=\sup_{v\in S^{n-1}}\frac{\rho_L(v)}{\rho_K(v)}\le \frac{M_1}{c_1}<\infty.
$$
Let $u\in S^{n-1}$.  Then, using (\ref{rlc}) and the fact that $\varphi_1$ and $\varphi_2$ are increasing, we obtain
$$
1-\ee\varphi_2(a_1)\le \varphi_1\left(\frac{\rho_K(u)}{\rho_{K\widetilde{+}_{\varphi,\ee}L}(u)}
\right)
.
$$
If $\ee$ is small enough, then $1-\ee \varphi_2(a_1)>0$ and hence
$
\rho_{K\widetilde{+}_{\varphi,\ee}L}(u)\varphi_1^{-1}\left(1-\ee\varphi_2(a_1)\right)\le \rho_K(u)
$.  It follows that
\begin{equation}\label{bound1}
0\le
\rho_{K\widetilde{+}_{\varphi,\ee}L}(u)-\rho_K(u)
\le M_1\left(1-\varphi_1^{-1}\left(1-\ee\varphi_2(a_1)\right)\right).
\end{equation}
Since $\varphi_1(1)=1$ and $\varphi_1$ is strictly increasing and continuous, the same is true for $\varphi_1^{-1}$, and therefore $\rho_{K\widetilde{+}_{\varphi,\ee}L}\to \rho_{K}$ uniformly on $S^{n-1}$ as $\ee\to 0+$.

Suppose that $\varphi_j\in {\Psi}_1^{(1)}$, $j=1,2$, and let $\ee\in (0,1]$. If $K,L\in {\mathcal{S}}_{+}^n$, then we now have $K\widetilde{+}_{\varphi,\ee}L\subset K$. Let $M_2<\infty$ be such that $K\subset M_2B^n$ and define
$$a_2=\inf_{v\in S^{n-1}}\frac{\rho_L(v)}{\rho_K(v)}\ge \frac{c_2}{M_2}>0.
$$
Let $u\in S^{n-1}$.  Then, using (\ref{rlc}) and the fact that $\varphi_1$ and $\varphi_2$ are decreasing, we obtain
$$
1-\ee\varphi_2(a_2)\le \varphi_1\left(\frac{\rho_K(u)}{\rho_{K\widetilde{+}_{\varphi,\ee}L}(u)}
\right).
$$
If $\ee$ is small enough, then $1-\ee \varphi_2(a_2)>0$ and using again that $\varphi_1$ is decreasing, we get
$
\rho_{K\widetilde{+}_{\varphi,\ee}L}(u)\varphi_1^{-1}\left(1-\ee\varphi_2(a_2)\right)\ge \rho_K(u)
$.
This implies
\begin{equation}\label{bound2}
0\le\rho_K(u)-\rho_{K\widetilde{+}_{\varphi,\ee}L}(u)
\le  M_2\left(\varphi_1^{-1}\left(1-\ee\varphi_2(a_2)\right)-1\right)
\end{equation}
and we draw the desired conclusion as before.
\end{proof}

\begin{rem}\label{remarknew1}
{\em
A minor variation of the preceding argument shows the following.
Suppose that either $\varphi_j\in {\Phi}_1^{(1)}$, $j=1,2$, $K\in {\mathcal{S}}^n_+$, and $L\in {\mathcal{S}}^n$, or $\varphi_j\in {\Psi}_1^{(1)}$, $j=1,2$, and $K,L\in {\mathcal{S}}_{+}^n$. Then $\rho_{K\widetilde{+}_{\varphi,\ee}L}\to \rho_{K}$ pointwise on $S^{n-1}$ as $\ee\to 0+$.
}
\end{rem}

\begin{lem}\label{limiteqns}
Let $\varphi_j\in {\Phi}_1^{(1)}$, $j=1,2$, let $K\in {\mathcal{S}}_{+}^n$, and let $L\in {\mathcal{S}}^n$.  If $(\varphi_1)'_l(1)$ exists and is positive, then
\begin{equation}\label{newlimits}
(\varphi_1)'_l(1)\lim_{\ee\rightarrow 0+}\frac{\rho_{K\widetilde{+}_{\varphi,\ee}L}(u)-\rho_K(u)}{\ee}=
\rho_K(u)\varphi_2\left(\frac{\rho_L(u)}{\rho_K(u)}\right),
\end{equation}
for $u\in S^{n-1}$. If $\varphi_j\in {\Psi}_1^{(1)}$, $j=1,2$, $K, L\in {\mathcal{S}}_{+}^n$,  and $(\varphi_1)'_r(1)$ exists and is positive, then \eqref{newlimits} holds with $(\varphi_1)'_l(1)$ replaced by $(\varphi_1)'_r(1)$.
\end{lem}

\begin{proof}
Suppose that $\varphi_j\in {\Phi}_1^{(1)}$, $j=1,2$.  Let $u\in S^{n-1}$ and let $\ee>0$. If $\rho_L(u)=0$, then $\rho_{K\widetilde{+}_{\varphi,\ee}L}(u)=\rho_K(u)$ and both sides of \eqref{newlimits} are zero. Suppose that $\rho_L(u)>0$.  Then $\rho_K(u)/\rho_{K\widetilde{+}_{\varphi,\ee}L}(u)< 1$ and we also have $\rho_K(u)>0$ and $\varphi_2({\rho_L(u)}/{\rho_K(u)})>0$. Using these facts, Remark~\ref{remarknew1}, and (\ref{rlc}), we obtain
\begin{align*}
(\varphi_1)'_l(1) &= \lim_{\ee\rightarrow 0+}\frac{1-\varphi_1 \left(\frac{\rho_K(u)}{\rho_{K\widetilde{+}_{\varphi,\ee}L}(u)}\right)}
{1-\frac{\rho_K(u)}{\rho_{K\widetilde{+}_{\varphi,\ee}L}(u)}}
=\lim_{\ee\rightarrow 0+} \rho_{K\widetilde{+}_{\varphi,\ee}L}(u)\frac{\ee\varphi_2 \left(\frac{\rho_L(u)}{\rho_{K\widetilde{+}_{\varphi,\ee}L}(u)}\right)}
{\rho_{K\widetilde{+}_{\varphi,\ee}L}(u)-\rho_K(u)}\\
&=\rho_K(u)\varphi_2\left(\frac{\rho_L(u)}{\rho_K(u)}\right)
\lim_{\ee\rightarrow 0+} \frac{\ee}{\rho_{K\widetilde{+}_{\varphi,\ee}L}(u)-\rho_K(u)},
\end{align*}
which yields (\ref{newlimits}).

If $\varphi_j\in {\Psi}_1^{(1)}$, $j=1,2$, then $\rho_K(u)/\rho_{K\widetilde{+}_{\varphi,\ee}L}(u)> 1$ and as in the previous paragraph, we conclude that \eqref{newlimits} holds with $(\varphi_1)'_l(1)$ replaced by $(\varphi_1)'_r(1)$.
\end{proof}

\begin{thm}\label{maindov}
Suppose that $\varphi_j\in {\Phi}_1^{(1)}$, $j=1,2$, $K,L\in {\mathcal{S}}^n$, and $c_1B^n\subset K$ for some $c_1>0$.  If $(\varphi_1)'_l(1)$ exists and is positive, then
\begin{equation}\label{ndov}
(\varphi_1)'_l(1)\lim_{\ee\rightarrow 0+}\frac{V_n(K\widetilde{+}_{\varphi,\ee}L)-V_n(K)}{\ee}=
\int_{S^{n-1}}
\varphi_2\left(\frac{\rho_L(u)}{\rho_K(u)}\right)\rho_K(u)^n\,du.
\end{equation}
Suppose instead that $\varphi_j\in {\Psi}_1^{(1)}$, $j=1,2$, $K, L\in {\mathcal{S}}_{+}^n$, and $c_2B^n\subset L$ for some $c_2>0$.  If $(\varphi_1)'_r(1)$ exists and is positive, then \eqref{ndov} holds with $(\varphi_1)'_l(1)$ replaced by $(\varphi_1)'_r(1)$.
\end{thm}

\begin{proof}
Let $\varphi_j\in {\Phi}_1^{(1)}$, $j=1,2$, let $K,L\in {\mathcal{S}}^n$, and let $c_1B^n\subset K$ for some $c_1>0$.   By the polar coordinate formula for volume, we have
$$
(\varphi_1)'_l(1)\left.\frac{d}{d\ee}\right|_{\ee=0+} V_n(K\widetilde{+}_{\varphi,\ee}L)=(\varphi_1)'_l(1)
\frac{1}{n}\lim_{\ee\to 0+}\int_{S^{n-1}}\frac{1}{\ee}\left(\rho_{K\widetilde{+}_{\varphi,\ee}L}(u)^n-\rho_K(u)^n\right)\, du.
$$
For $\ee>0$ sufficiently small and $u\in S^{n-1}$, using the notation from the proof of Lemma~\ref{newconvlem}, we conclude from \eqref{bound1} that
\begin{equation}\label{upperb1}
0\le \frac{1}{\ee}\left(\rho_{K\widetilde{+}_{\varphi,\ee}L}(u)-\rho_K(u)\right)\le M_1\frac{1}{\ee}\left(1-\varphi_1^{-1}(1-\ee \varphi_2(a_1)\right).
\end{equation}
Since $(\varphi_1)'_l(1)>0$ and $\varphi_1$ is strictly increasing,  we also have $(\varphi_1^{-1})'_l(1)>0$, so the right-hand side of \eqref{upperb1} is bounded above uniformly in $\ee>0$. Hence we can apply Lebesgue's dominated convergence theorem and Lemma~\ref{limiteqns} to obtain
\begin{align*}
(\varphi_1)'_l(1)\left.\frac{d}{d\ee}\right|_{\ee=0+} V_n(K\widetilde{+}_{\varphi,\ee}L)
&=(\varphi_1)'_l(1)\frac{1}{n}\int_{S^{n-1}}
\left.\frac{d}{d\ee}\right|_{\ee=0+}\rho_{K\widetilde{+}_{\varphi,\ee}L}(u)^n\, du\\
&=(\varphi_1)'_l(1)\frac{1}{n}\int_{S^{n-1}}
\left.\frac{d}{d\ee}\right|_{\ee=0+}\rho_{K\widetilde{+}_{\varphi,\ee}L}(u) n\rho_K(u)^{n-1}\, du\\
&=\int_{S^{n-1}}\varphi_2\left(\frac{\rho_{L}(u)}{\rho_{K}(u)}\right)
\rho_{K}(u)^{n}\, du.
\end{align*}

The second part of the theorem follows in the same way if \eqref{bound2} is used instead of \eqref{bound1}.
\end{proof}

Let $\varphi:(0,\infty)\to \R$.  The formula (\ref{ndov}) for the first variation of volume suggests defining the {\em dual Orlicz mixed volume} of star sets $K,L\in\mathcal{S}^n_+$  by
\begin{equation}\label{dualOmv}
\widetilde{V}_\varphi(K,L)=\frac{1}{n}\int_{S^{n-1}}\varphi\left(\frac{\rho_{L}(u)}
{\rho_{K}(u)}\right)\rho_{K}(u)^{n}\,du=V_n(K)\int_{S^{n-1}}\varphi\left(\frac{\rho_L(u)}{\rho_K(u)}\right)\, d\widetilde{V}_K(u),
\end{equation}
whenever these expressions make sense, where integration on the right is with respect to the dual cone measure of $K$, defined by (\ref{dualcone}). A similar
remark applies if $\varphi:[0,\infty)\to\R$ and $K,L$ belong to an appropriate class of star sets.

For example, in the important special case when $\varphi(t)=t^p$, $p\in \R\setminus\{0\}$, the dual Orlicz mixed volume $\widetilde{V}_\varphi(K,L)$ becomes
\begin{equation}\label{dmvp}
\widetilde{V}_p(K,L)=\frac{1}{n}\int_{S^{n-1}}\rho_K(u)^{n-p}
\rho_L(u)^p\,du,
\end{equation}
where the latter quantity is defined as in \cite[(A.56), p.~410]{Gar06} (with $i=p$).  Here it would be assumed that $K\in {\cS}^n$ contains a ball with positive radius if $p>n$ and that $L\in {\cS}^n$ contains a ball with positive radius if $p<0$.

\section{Dual Orlicz-Minkowski inequalities}\label{dOMi}

In this section, we prove some dual Orlicz-Minkowski inequalities and corollaries thereof.

\begin{thm}\label{maindom}
Let $\varphi:(0,\infty)\to \R$ be such that $\varphi_0(t)=\varphi(t^{1/n})$, $t>0$, is concave.  If $K,L\in {\cS}_{c+}^n$, then
$$
\widetilde{V}_\varphi (K,L)\le V_n(K)\,\varphi\left(\left( \frac{V_n(L)}{V_n(K)} \right)^{1/n}\right),
$$
while if $\varphi_0$ is convex, the inequality is reversed.  If $\varphi_0$ is strictly concave (or convex, as appropriate), equality holds if and only if $K$ and $L$ are dilatates.
\end{thm}

\begin{proof}
By (\ref{dualOmv}) and Jensen's inequality for concave functions (the reverse of (\ref{JenIneq})), we obtain
\begin{align*}
\widetilde{V}_\varphi(K,L)&=V_n(K)\int_{S^{n-1}}\varphi\left(\frac{\rho_L (u)}{\rho_K( u)}\right)\, d\widetilde{V}_K(u)\\
&=V_n(K)\int_{S^{n-1}}\varphi_0\left(\left(\frac{\rho_L(u)}{\rho_K(u)}
\right)^n\right)\, d\widetilde{V}_K(u)\\
&\le V_n(K)\,\varphi_0\left(\int_{S^{n-1}}\left(\frac{\rho_L(u)}{\rho_K(u)}
\right)^n\,d\widetilde{V}_K(u)\right)\\
&=V_n(K)\,\varphi_0\left(\int_{S^{n-1}}\frac{\rho_L(u)^n}{nV_n(K)}\, du\right)\\
&= V_n(K)\,\varphi_0\left(\frac{V_n(L)}{V_n(K)}\right)=
V_n(K)\,\varphi\left(\left(\frac{V_n(L)}{V_n(K)}\right)^{1/n}\right).
\end{align*}
If equality holds, then equality also holds in Jensen's inequality.  The assumption that $\varphi_0$ is strictly concave and the equality condition for Jensen's inequality imply that $\rho_{L}(u)/\rho_{K}(u)$ is constant on $S^{n-1}$.  This establishes the equality condition.

The remainder of the theorem follows in the same fashion.
\end{proof}

Note that taking $\varphi(t)=t^p$, $p\in \R$, in Theorem~\ref{maindom}, we retrieve the {\em dual $L_p$-Minkowski inequality}
\begin{equation}\label{duallpM}
\widetilde{V}_p(K,L)^n\le V_n(K)^{n-p}V_n(L)^p,
\end{equation}
and its equality conditions, where $p\in [0,n]$ and $\widetilde{V}_p(K,L)$ is defined by (\ref{dmvp}), and where the inequality in (\ref{duallpM}) is reversed if $p<0$ or $p>n$. See \cite[(B.29), p.~422]{Gar06} (with $i=p$).  We can also obtain the following result, which follows directly from Theorem~\ref{maindom} on taking $\varphi(t)=\log t$.

\begin{thm}\label{starlog}
If $K,L\in {\cS}_{c+}^n$, then
$$
\int_{S^{n-1}}\log\left(\frac{\rho_L(u)}{\rho_K(u)}\right)\, d\widetilde{V}_{K}(u)\le\log\left(\left(\frac{V_n(L)}{V_n(K)}\right)^{1/n}
\right),
$$
with equality if and only if $K$ and $L$ are dilatates.
\end{thm}

It is remarkable that the previous inequality is precisely the dual of the conjectured (and so far proved only for $n=2$) $\log$-Minkowski inequality (see \cite[p.~1976]{BLYZ}):
$$
\int_{S^{n-1}}\log\left(\frac{h_L(u)}{h_K(u)}\right)\, d\overline{V}_{K}(u)\ge\log\left(\left(\frac{V_n(L)}{V_n(K)}\right)^{1/n}
\right).
$$
Here $h_K$ and $h_L$ are the support functions of centrally symmetric convex bodies $K$ and $L$ in $\R^n$ and integration is with respect to the Borel probability measure in $S^{n-1}$ defined by
$$
d\overline{V}_K(u)=\frac{h_K(u)}{nV_n(K)}\,dS(K,u),
$$
where $S(K,\cdot)$ is the surface area measure of $K$.  This measure is called the normalized cone measure (or cone-volume probability measure) of $K$.

\begin{cor}\label{polarcor}
Let $K$ and $L$ be convex bodies in $\R^n$ containing the origin in their interiors. Then
$$
\int_{S^{n-1}}\log\left(\frac{h_{L}(u)}{h_K(u)}\right)\, d\widetilde{V}_{K^\circ}(u)\ge\log\left(\left(\frac{V_n(K^\circ)}{V_n(L^\circ)}
\right)^{1/n}\right),
$$
with equality if and only if $K$ and $L$ are dilatates.
\end{cor}

\begin{proof}
In Theorem~\ref{starlog}, replace $K$ and $L$ by the polar bodies $K^\circ$ and $L^\circ$, respectively, and use the relation $\rho_{K^\circ}(u)=1/h_K(u)$, $u\in S^{n-1}$ (see \cite[(0.36), p.~20]{Gar06}).
\end{proof}

We end this section by remarking that Theorems~\ref{dualOBMI1} and~\ref{maindom} are related, as follows.  Firstly, Theorem~\ref{maindom} implies the important special case of Theorem~\ref{dualOBMI1} when $\varphi$ is defined by (\ref{spphi}).  To see this, one applies Theorem~\ref{maindom} with $K$, $L$, and $\varphi$ replaced by $K_1\widetilde{+}_{\varphi}\cdots \widetilde{+}_{\varphi}K_m$, $K_j$, and $\varphi_j$, respectively, for $j=1,\dots,m$, following the analogous argument in the remark after \cite[Theorem~9.2]{GHW2}.  On the other hand, Theorem~\ref{dualOBMI1} implies Theorem~\ref{maindom} if it is assumed in addition that $\varphi\in \Phi_1^{(1)}\cup \Psi_1^{(1)}$.  This can be seen by applying Theorem~\ref{dualOBMI1} with $m=2$, $K_1=K$, and $K_2=L$, to the function $\varphi(x_1)+\varphi(x_2)$, and following the analogous argument in \cite[pp.~370--371]{XJL}.

\section{Radial $M$-addition}\label{radialMadd}

For an arbitrary set $M\subset [0,\infty)^m$ and star sets $K_j\in {\cS}^n$, $j=1,\dots,m$, define the {\em radial $M$-sum} of $K_1,\dots,K_m$ by
\begin{equation}\label{radialMdef}
\widetilde{\oplus}_M(K_1,\dots,K_m)=\cup_{(\alpha_1,\dots,\alpha_m)\in M}
\,\alpha_1K_1\widetilde{+}\cdots\widetilde{+}\alpha_mK_m,
\end{equation}
where the radial linear combination in the union is defined as in (\ref{radials2}). This definition results from that of $M$-addition (see \cite[(25)]{GHW2} and the equivalent definition given immediately after it) when the Minkowski addition there is replaced by radial addition.

For each fixed $(\alpha_1,\dots,\alpha_m)\in M$, the radial linear combination $\alpha_1K_1\widetilde{+}\cdots\widetilde{+}\alpha_mK_m\in {\cS}^n$, but in general $\widetilde{\oplus}_M(K_1,\dots,K_m)\not\in {\cS}^n$ since it may not be a bounded Borel set.  However, if $M$ is compact, then $\widetilde{\oplus}_M(K_1,\dots,K_m)\in {\cS}^n$, so we shall make this assumption about $M$ from now on.

The formula
\begin{equation}\label{M1}
\rho_{\widetilde{\oplus}_M(K_1,\dots,K_m)}(x)=h_{\conv M}\left(\rho_{K_1}(x),\dots,
\rho_{K_m}(x)\right),
\end{equation}
holds for $x\in \R^n\setminus\{o\}$, where $h_{\conv M}$ denotes the support function of the convex hull of $M$.  Indeed,
\begin{eqnarray*}
\rho_{\widetilde{\oplus}_M(K_1,\dots,K_m)}(x)&=
&\rho_{\cup_{(\alpha_1,\dots,\alpha_m)\in M}
\,\alpha_1K_1\widetilde{+}\cdots\widetilde{+}\alpha_mK_m}(x)\\
&=& \max\{\rho_{
\alpha_1K_1\widetilde{+}\cdots\widetilde{+}\alpha_mK_m}(x):
(\alpha_1,\dots,\alpha_m)\in M\}\\
&=& \max\{\alpha_1\rho_{K_1}(x)+\cdots+\alpha_m \rho_{K_m}(x):(\alpha_1,\dots,\alpha_m)\in M\}\\
&=& \max\{(\alpha_1,\dots,\alpha_m)\cdot\left(\rho_{K_1}(x),\dots, \rho_{K_m}(x)\right):(\alpha_1,\dots,\alpha_m)\in M\}\\
&=&h_{\conv M}\left(\rho_{K_1}(x),\dots,
\rho_{K_m}(x)\right).
\end{eqnarray*}
In particular, if $M$ is also convex, we have
\begin{equation}\label{M2}
\rho_{\widetilde{\oplus}_M(K_1,\dots,K_m)}(x)=h_{M}\left(\rho_{K_1}(x),\dots,
\rho_{K_m}(x)\right),
\end{equation}
for $x\in \R^n\setminus\{o\}$.

It may be surprising that (\ref{M1}) and (\ref{M2}) involve a support function.  But note that if $M\in {\cS}^n$, then the function
$$\rho_{M}\left(\rho_{K_1}(x),\dots,\rho_{K_m}(x)\right),$$
for $x\in \R^n\setminus\{o\}$ (the dual analog of \cite[Theorem 6.5(ii)]{GHW}, for example), does not define a radial function since it is homogeneous of degree 1 and not $-1$.

When $m=2$, the radial $M$-sum of $K, L\in {\cS}^n$ is denoted by $K\widetilde{\oplus}_M L$.  Note that in this case, if $\{(1,1)\}\subset M\subset [0,1]^2$, then $\widetilde{\oplus}_M$ is ordinary radial addition, and if
\begin{equation}\label{mmm}
M=\left\{(a,b)\in [0,1]^2:a^{p'}+b^{p'}=1\right\}=\left\{\left((1-t)^{1/p'},t^{1/p'}\right): 0\le t\le 1\right\},
\end{equation}
where $p>1$ and $1/p+1/p'=1$ (or $\conv M$ or $\conv\{M,o\}$ with $M$ as in (\ref{mmm})), then $\widetilde{\oplus}_M$ is $p$th radial addition.

From either (\ref{radialMdef}) or (\ref{M1}), it can be verified that $\widetilde{\oplus}_M:\left({\cS}^n\right)^m\to {\cS}^n$ is homogeneous of degree 1, $GL(n)$ covariant, section covariant, and continuous in the sense of pointwise convergence of radial functions.  It does not in general have the identity property.

For the next result, recall that a set in $\R^n$ is called {\em $1$-unconditional} if it is symmetric with respect to each coordinate hyperplane.

\begin{thm}\label{orlmadd}
If $M$ is a $1$-unconditional convex body in $\R^m$, then there is a convex function $\varphi\in \Phi_m$ such that
\begin{equation}\label{morl} \widetilde{\oplus}_{M\cap [0,\infty)^m}(K_1,\dots,K_m)=\widetilde{+}_{\varphi}(K_1,\dots,K_m),
\end{equation}
for all $K_j\in {\cS}^n$, $j=1,\dots,m$.

Conversely, given a convex $\varphi\in \Phi_m$, there is a $1$-unconditional convex body $M$ in $\R^n$ such that \eqref{morl} holds for all $K_j\in {\cS}^n$, $j=1,\dots,m$.
\end{thm}

\begin{proof}
Since $M$ is $1$-unconditional, its polar body $M^{\circ}$ is also.  Then, by an easy modification of \cite[Theorem~5.4]{GHW2} (with $K$ there replaced by $M^{\circ}$), there is a convex function $\varphi\in \Phi_m$ such that the part of the boundary of $M^{\circ}$ contained in $[0,\infty)^m$ is given by $\{(x_1,\dots,x_m)\in [0,\infty)^m: \varphi(x_1,\dots,x_m)=1\}$.  Let $J_{\varphi}$ be the $1$-unconditional convex body in $\R^m$ defined by
\begin{equation}\label{jphi}
J_{\varphi}\cap [0,\infty)^m=\{(x_1,\dots,x_m)\in [0,\infty)^m:\varphi(x_1,\dots,x_m)\le 1\}.
\end{equation}
The argument in the first paragraph of the proof of \cite[Theorem~5.3]{GHW2} shows that $J_{\varphi}$ is indeed a $1$-unconditional convex body.  Then we have $M^{\circ}=J_{\varphi}$ and hence $M=J_{\varphi}^{\circ}$. Let $x\in\R^n\setminus\{o\}$ and let $K_j\in {\cS}^n$, $j=1,\dots,m$, be such that $\rho_{K_1}(x)+\cdots+\rho_{K_m}(x)>0$.  By (\ref{M2}),
\begin{eqnarray*}
\rho_{\widetilde{\oplus}_{M\cap [0,\infty)^m}(K_1,\dots,K_m)}(x)&=&h_{M\cap [0,\infty)^m}\left(\rho_{K_1}(x),
\dots,\rho_{K_m}(x)\right)=h_{M}\left(\rho_{K_1}(x),
\dots,\rho_{K_m}(x)\right)\\
&=&h_{J_{\varphi}^{\circ}}\left(\rho_{K_1}(x),\dots,
\rho_{K_m}(x)\right)=\rho_{J_{\varphi}}\left(\rho_{K_1}(x),\dots,
\rho_{K_m}(x)\right)^{-1}.
\end{eqnarray*}
By the definition of the radial function, $\rho_{J_{\varphi}}\left(\rho_{K_1}(x),\dots,
\rho_{K_m}(x)\right)$ is the number $c$ such that
$$c\left(\rho_{K_1}(x),\dots,
\rho_{K_m}(x)\right)\in\partial J_{\varphi}.$$  But this implies that $\varphi\left(c\rho_{K_1}(x),\dots,c\rho_{K_m}(x)\right)=1$ and hence, from (\ref{Orldef}),
\begin{equation}\label{eqrhos}
\rho_{\widetilde{+}_{\varphi}(K_1,\dots,K_m)}(x)=c^{-1}=
\rho_{\widetilde{\oplus}_{M\cap [0,\infty)^m}(K_1,\dots,K_m)}(x).
\end{equation}
If $\rho_{K_1}(x)=\cdots=\rho_{K_m}(x)=0$, the equality of the two radial functions in (\ref{eqrhos}) holds trivially. Therefore (\ref{morl}) holds.

Conversely, suppose that $\varphi\in \Phi_m$ is a convex function. Let $J_{\varphi}$ be the $1$-unconditional convex body in $\R^m$ defined by (\ref{jphi}).  Let $x\in\R^n\setminus\{o\}$ and let $K_j\in {\cS}^n$, $j=1,\dots,m$, be such that $\rho_{K_1}(x)+\cdots+\rho_{K_m}(x)>0$.  If $c>0$ is such that $\rho_{\widetilde{\oplus}_{M\cap [0,\infty)^m}(K_1,\dots,K_m)}(x)=1/c$, the steps in the previous paragraph can be reversed to conclude that (\ref{eqrhos}) and hence (\ref{morl}) holds with $M=J_{\varphi}^{\circ}$.
\end{proof}

The second statement in Theorem~\ref{orlmadd} is not true in general if $\varphi$ is not a convex function in $\Phi_m$.  For example, suppose that $m=2$ and $\varphi(s,t)=\sqrt{s}+\sqrt{t}$ for $s,t\ge 0$.  From (\ref{Orldef}), we get
$$\rho_{K\widetilde{+}_{\varphi}L}(x)=\left(\sqrt{\rho_K(x)}+
\sqrt{\rho_L(x)}\right)^2,$$
for all $x\in \R^n\setminus\{o\}$ and $K,L\in{\cS}^n$.  Comparing with (\ref{M1}), we see that if $\widetilde{+}_{\varphi}$ is a radial $M$-addition on the class of star sets, we would have
$$h_{\conv M}\left(\rho_{K}(x),\rho_L(x)\right)=\left(\sqrt{\rho_K(x)}+
\sqrt{\rho_L(x)}\right)^2,$$
for $x\in \R^n\setminus\{o\}$ and all $K,L\in{\cS}^n$. Applying this equation with $K=sB^n$ and $L=tB^n$ for $s,t\ge 0$, we conclude that
$$h_{\conv M}(s,t)=\left(\sqrt{s}+\sqrt{t}\right)^2,$$
for $s,t\ge 0$.  However, the function on the right-hand side is not sublinear and hence not a support function, so $\widetilde{+}_{\varphi}$ is not a radial $M$-addition for any $M$.

\section*{Appendix}\label{Appendix}

The present paper is a combination of a manuscript by the first three authors dated May~21, 2013, and the preprint arXiv:1404.6991 written independently by the fourth author.  After the latter article was completed, the independent work \cite{ZZ} (see also \cite{ZZX}) came to the attention of the fourth author, who thanks B.~Zhu for communicating it.

In \cite{ZZ}, only functions $\varphi:(0,\infty)^2\to (0,\infty)$ of the form $\varphi(x_1,x_2)=\alpha_1\phi(x_1)+\alpha_2\phi(x_2)$, where $\phi$ is a convex strictly decreasing function on $(0,\infty)$, are considered.  Moreover, for the most part \cite{ZZ} deals only with the class ${\cS}^n_{c+}$.  The corresponding special cases of Theorems~\ref{dualOBMI1} and~\ref{maindom} are proved in \cite{ZZ}, but only under the rather stronger condition that the function $\phi$ is convex.

\end{document}